\documentclass{amsart}

\usepackage{mathrsfs}
\usepackage{graphicx,amssymb,xcolor,enumerate}
\usepackage{hyperref}
\usepackage{verbatim}
\hypersetup{
    bookmarks=true,         
    pdffitwindow=false,     
    pdfstartview={FitH},    
    colorlinks=false,       
}
 \newtheorem{thm}{Theorem}[section]
 
 \newtheorem{lem}[thm]{Lemma}
 \newtheorem{prop}[thm]{Proposition}
 \theoremstyle{definition}
 \newtheorem{defn}[thm]{Definition}
 \theoremstyle{remark}
 \newtheorem{rem}[thm]{Remark}
 
 \numberwithin{equation}{section}

\usepackage{graphicx}
\usepackage{amssymb, amsmath, amsthm}
\usepackage{amsfonts}
\usepackage{amssymb}
\usepackage{float}
\usepackage{mathrsfs}
\usepackage{enumitem} 
\newtheorem{theorem}{Theorem}

\newtheorem{proposition}[theorem]{Proposition}

\newcommand{\trho}{\rho_s }
\newcommand{\dis}{\displaystyle}

\newcommand{\divv}{\text{\rm div}}

\newcommand{\R}{\mathbb R}

\newcommand{\intox}{\int_{\R^3}}

\newcommand{\intoxt}{\int_0^T\int_{\R^3}}
\newcommand{\supp}{\text{\rm supp }}
\newcommand{\intoxts}{\int_0^{t_2}\intox}

\numberwithin{equation}{section}
\numberwithin{theorem}{section}

\numberwithin{figure}{section}

\begin{document}

%
%
%
%
%
%
%
%
%








\title[Uniqueness of compressible flow with potential force]
 {Uniqueness of weak solutions of the three-dimensional compressible Navier-Stokes equations with potential force}
 
\author{Anthony Suen} 

\address{Department of Mathematics and Information Technology\\
The Education University of Hong Kong}

\email{acksuen@eduhk.hk}

\date{\today}



\begin{abstract}
We prove uniqueness of weak solutions of the three-dimensional compressible Navier-Stokes equations with potential force. We make use of the Lagrangean framework in comparing the instantaneous states of corresponding fluid particles in two different solutions. The present work provides qualitative results on how the weak solutions depend continuously on initial data and steady states.
\end{abstract}

\maketitle
\section{Introduction}

We are interested in the 3-D compressible Navier-Stokes equations with an external potential force in the whole space $\R^3$ ($j=1,2,3$):
\begin{align}\label{NS} 
\left\{ \begin{array}{l}
\rho_t + \divv (\rho u) =0, \\
(\rho u^j)_t + \divv (\rho u^j u) + (P)_{x_j} = \mu\,\Delta u^j +
\lambda \, (\divv \,u)_{x_j}  + \rho f^{j}.
\end{array}\right.
\end{align}
Here $x\in\R^3$ is the spatial coordinate and $t\ge0$ stands for the time. The unknown functions $\rho=\rho(x,t)$ and $u=(u^1,u^2,u^3)(x,t)$ represent the density and velocity vector in a compressible fluid. The function $P=P(\rho)$ denotes the pressure, $f=(f^1(x),f^2(x),f^3(x))$ is a prescribed external force and $\mu$, $\lambda$ are positive viscosity constants. The system \eqref{NS} is equipped with initial condition
\begin{equation}\label{IC}
(\rho(\cdot,0)-\trho,u(\cdot,0)) = (\rho_0-\trho,u_0),
\end{equation}
where the non-constant time-independent function $\rho_s =\rho_s (x)$ (known as the {\it steady state solution} to \eqref{NS}) can be obtained formally by taking $u\equiv0$ in \eqref{NS}:
\begin{align}\label{steady state}
\nabla P(\rho_s (x)) =\rho_s (x)f(x).
\end{align}

The well-posedness problem of the Navier-Stokes system \eqref{NS} is an important but challenging research topic in fluid mechanics, and we now give a brief review on the related results. The local-in-time existence of classical solution to the full Navier-Stokes equations was proved by Nash \cite{nash} and Tani \cite{tani}, and some Serrin type blow-up criteria for smooth solutions was recently obtained by Suen \cite{suen20c}. Later, Matsumura and Nishida \cite{mn1} obtained the global-in-time existence of $H^3$ solutions when the initial data was taken to be small with respect to $H^3$ norm, the results were then generalised by Danchin \cite{danchin} who showed the global existence of solutions in critical spaces. In the case of large initial data, Lions \cite{lions} obtained the existence of global-in-time finite energy weak solutions, yet the problem of uniqueness for those weak solutions remains completely open. In between the two types of solutions as mentioned above, a type of ``intermediate weak'' solutions were first suggested by Hoff in \cite{hoff95, hoff02, hoff05, hoff06} and later generalised by Matsumura and Yamagata in \cite{MY01}, Suen in \cite{suen13b, suen14, suen16} and other systems which include compressible magnetohydrodynamics (MHD) \cite{suenhoff12, suen12, suen20b}, compressible Navier-Stokes-Poisson system \cite{suen20a} and chemotaxis systems \cite{LS16}. Solutions as obtained in this intermediate class are less regular than those small-smooth type solutions obtained by Matsumura and Nishida \cite{mn1} and Danchin \cite{danchin}, which contain possible codimension-one discontinuities in density, pressure, and velocity gradient. Nevertheless, those intermediate weak solutions would be more regular than the large-weak type solutions developed by Lions \cite{lions}, hence the uniqueness and continuous dependence of solutions may be obtained; see \cite{hoff06} and the compressible MHD \cite{suen20b}. 

In this present work, we compare solutions of \eqref{NS}-\eqref{IC} from the class of intermediate weak solutions as mentioned above. Our results extend that of Cheung and Suen \cite{suen16}, which proved uniqueness of weak solution to \eqref{NS}-\eqref{IC} with H\"{o}lder continuous density function $\rho$. Such condition on $\rho$ is too strong in the sense that this would exclude solutions with codimension-one singularities, which are physically interesting features for the weak solutions; see \cite{hoff02} for a detailed discussion on the propagation of singularities. The main novelties of this current work can be summarised as follows:
\begin{itemize}
\item We provide detailed descriptions on the global-in-time existence and regularity of weak solution to \eqref{NS}-\eqref{IC}, see Theorem~\ref{existence of weak sol thm} in section~\ref{prelim section};
\item We remove the H\"{o}lder continuity restriction on density from \cite{suen16} which allows us to include a larger class of weak solutions;
\item We give a more precise results on how the weak solutions to \eqref{NS}-\eqref{IC} depend continuously on initial data and steady state solutions, see \eqref{bound on difference} in Theorem~\ref{Main thm}.
\end{itemize} 
 
We now give a precise formulation of our results. For $r\in(1,\infty]$, we define the following function spaces:
\begin{align*} 
\left\{ \begin{array}{l}
L^r=L^r(\R^3), D^{k,r}=\{u\in L^1_{loc}(\R^3):\|\nabla^k u\|_{L^r}<\infty\},\|u\|_{D^{k,r}}:=\|\nabla^k u\|_{L^r}\\
W^{k,r}=L^r\cap D^{k,r}, H^k=W^{k,2}.
\end{array}\right.
\end{align*}
We introduce the usual convective derivative $\frac{d}{dt}$ with respect to a velocity field $u$ as follows. For a given function $w:\R^3\times(0,T)\to\R$, we define
\begin{align}
\frac{d}{dt}(w)=\dot{w}:=w_t+u\cdot\nabla w,
\end{align}
where $\nabla w$ is the gradient of $w$. For $w:\R^3\times(0,T)\to\R^3$, we define
\begin{align}
\frac{d}{dt}(w)=\dot{w}:=w_t+\nabla w u,
\end{align}
where $\nabla w$ is the $3 \times 3$ matrix of partial derivatives of $w$.

We define the system parameters $P$, $f$, $\mu$, $\lambda$ as follows. For the pressure function $P=P(\rho)$ and the external force $f$, we assume that
\begin{align}\label{condition on P}
\mbox{$P(\rho)=a\rho$ with $a>0$;}
\end{align}
\begin{align}\label{condition on f}
\mbox{There exists $\psi\in H^2$ such that $f=\nabla\psi$ and $\psi(x)\to 0$ as $|x|\to\infty$.}
\end{align}
The viscosity coefficients $\mu$ and $\lambda$ are assumed to satisfy
\begin{align}\label{assumption on viscosity}
\lambda\ge0,\qquad \mu>0.
\end{align}

Next, we define $\rho_s $ as mentioned at the beginning of this section. Given a constant densty $\rho_{\infty}>0$, we say $(\rho_s ,0)$ is a steady state solution to \eqref{NS} if $\rho_s \in C^2(\R^3)$ and the following holds
\begin{align}
\label{eqn for steady state} \left\{ \begin{array}{l}
\nabla P(\rho_s (x)) =\rho_s (x)\nabla\psi(x), \\
\lim\limits_{|x|\rightarrow\infty}\rho_s (x) = \rho_{\infty}.
\end{array}\right.
\end{align}
By solving \eqref{eqn for steady state}, $\rho_s$ can be expressed explicitly as follows:
\begin{align}\label{expression for rho s}
\rho_s(x)=\rho_{\infty}\exp(\frac{1}{a}\psi(x)).
\end{align}
From now on, we choose $\rho_\infty\equiv1$ and take $\trho$ which satisfies \eqref{expression for rho s}. We also write $P_s=P(\rho_s)$ for simplicity.

We further introduce two important functions, namely the {\it effective viscous flux} $F$ and {\it vorticity} $\omega$, which are defined by
\begin{equation}\label{def of F and omega}
\rho_s  F=(\mu+\lambda)\divv\,u-(P(\rho)-P(\rho_s )),\qquad\omega=\omega^{j,k}=u^j_{x_k}-u^k_{x_j}.
\end{equation}
By the definitions of $F$ and $\omega$, and together with \eqref{NS}$_2$, $F$ and $\omega$ satisfy the elliptic equations
\begin{align}\label{elliptic eqn for F}
\Delta (\rho_s  F)=\divv(\rho \dot{u}-\rho f+\nabla P(\rho_s )),
\end{align}
\begin{align}\label{elliptic eqn for omega}
\mu\Delta\omega=\nabla\times(\rho \dot{u}-\rho f+\nabla P(\rho_s )).
\end{align}
The functions $F$ and $\omega$ play essential roles for studying intermediate weak solutions to compressible flows, see \cite{hoff95, hoff02, suenhoff12} for more detailed discussions.

Weak solutions to the system \eqref{NS}-\eqref{IC} can be defined as follows. We say that $(\rho,u,f,\rho_s)$ on $\R^3\times[0,T]$ is a {\it weak solution} of \eqref{NS}-\eqref{IC} if the following conditions hold:
\begin{equation}\label{condition on rho s}
\mbox{$\rho_s$ is a steady state solution to \eqref{eqn for steady state} which satisfies \eqref{expression for rho s};}
\end{equation}
\begin{equation}\label{condition on weak sol1}
\mbox{$\rho-\rho_s$ is a bounded map from $[0,T]$ into $L^1_{loc}\cap H^{-1}$ and $\rho\ge0$ a.e.;}
\end{equation}
\begin{equation}\label{condition on weak sol2}
\rho_0u_0\in L^2; \rho u,P-P_s,\nabla u,\rho f\in L^2(\R^3\times(0,T));\rho|u|^2\in L^1(\R^3\times(0,T));
\end{equation}
For all $t_2\ge t_1 \ge 0$ and $C^1$ test functions $\varphi$ which are Lipschitz on $\R^3\times[t_1,t_2]$ with $\supp\varphi(\cdot,t)\subset K$, $t\in[t_1,t_2]$, where $K$ is compact and
\begin{align}\label{WF1}
\left.\int_{\R^3}\rho(x,\cdot)\varphi(x,\cdot)dx\right|_{t_1}^{t_2}=\int_{t_1}^{t_2}\int_{\R^3}(\rho\varphi_t + \rho u\cdot\nabla\varphi)dxdt;
\end{align}
The weak form of the momentum equation
\begin{align}\label{WF2}
\left.\int_{\R^3}(\rho u^{j})(x,\cdot)\varphi(x,\cdot)dx\right|_{t_1}^{t_2}=&\int_{t_1}^{t_2}\int_{\R^3}[\rho u^{j}\varphi_t + \rho u^{j}u\cdot\nabla\varphi + (P-P_s)\varphi_{x_j}]dxdt\notag\\
& - \int_{t_1}^{t_2}\int_{\R^3}[\mu\nabla u^{j}\cdot\nabla\varphi + (\mu - \xi)(\divv\,u)\varphi_{x_j}]dxdt\\
&+ \int_{t_1}^{t_2}\int_{\R^3}(\rho f - \nabla P_s)\cdot\varphi dxdt.\notag
\end{align}
holds for test functions $\varphi$ which are locally Lipschitz on $\R^3 \times [0, T]$ and for which $\varphi,\varphi_t,\nabla\varphi \in L^2(\R^3 \times (0,T))$, $\nabla\varphi \in L^\infty(\R^3 \times (0,T))$, and $\varphi(\cdot,T) = 0$.

For the two solutions $(\rho,u,f,\rho_s)$ and $(\bar\rho,\bar{u},\bar{f},\bar{\rho_s})$ we compare, they will be assumed to satisfy
\begin{equation}\label{spaces for weak sol 1}
u,\bar{u}\in C(\R^3\times(0,T])\cap L^1((0,T);W^{1,\infty})\cap L^\infty_{loc}((0,T];L^\infty);
\end{equation}
\begin{equation}\label{spaces for weak sol 2}
\rho-\rho_s,\bar{\rho}-\bar{\rho_s},u,\bar{u},f,\bar{f}\in L^2(\R^2\times(0,T)).
\end{equation}
One of the solutions $(\rho,u,f,\rho_s)$ will have to satisfy
\begin{equation}\label{condition on weak sol3}
\|f\|_{L^\infty}<\infty,
\end{equation}
\begin{equation}\label{condition on weak sol4}
\rho,\rho^{-1}\in L^\infty(\R^3\times(0,T)),
\end{equation}
and 
\begin{equation}\label{condition on weak sol5}
\int_0^T\intox|u|^r dxdt<\infty
\end{equation}
for some $r>3$, and the other solution $(\bar\rho,\bar{u},\bar{f},\bar{\rho_s})$ will have to satisfy
\begin{equation}\label{condition on weak sol6}
\|\bar{\rho_s}\|_{L^\infty}+\|\nabla\bar{\rho_s}\|_{L^\infty}<\infty,
\end{equation}
\begin{align}\label{condition on weak sol8}
&\int_0^T[t\|\nabla\bar{F}(\cdot,t)\|^2_{L^2}+t\|\nabla\bar{\omega}(\cdot,t)\|^2_{L^2}+t^\alpha\|\nabla\bar{F}(\cdot,t)\|^{2\alpha}_{L^4}+t^\alpha\|\nabla\bar{\omega}(\cdot,t)\|^{2\alpha}_{L^4}]dt\notag\\
&\qquad+\int_0^T[\|\bar{u}(\cdot,t)\|^2_{L^\infty}+t\|\nabla\bar{u}(\cdot,t)\|^2_{L^\infty})dt<\infty,
\end{align}
where $\bar{F}$ and $\bar{\omega}$ are as in \eqref{elliptic eqn for F}-\eqref{elliptic eqn for omega} and $\alpha = \frac{4}{5}$; and
\begin{equation}\label{condition on weak sol9}
\bar{f}\in L^{2q},
\end{equation}
for some $q\in[1,\infty]$. Finally, we assume that
\begin{equation}\label{condition on weak sol10}
\rho_0-\bar{\rho}_0\in L^2\cap L^{2p},
\end{equation}
where $p$ is the H\"{o}lder conjugate of $q$.

\medskip

We are ready to state the following main results which are given in Theorem~\ref{Main thm}:

\begin{thm}\label{Main thm}
Given $a>0$, let $P$, $f$, $\lambda$, $\mu$ be the system parameters in \eqref{NS} satisfying \eqref{condition on P}-\eqref{assumption on viscosity}. Given $M$, $T$ and $r>3$, there is a positive constant $C$ depending on $a$, $M$, $T$ and $r$ such that if $(\rho,u,f,\rho_s)$ and $(\bar\rho,\bar{u},\bar{f},\bar{\rho_s})$ are weak solutions of \eqref{NS} satisfying \eqref{spaces for weak sol 1}-\eqref{spaces for weak sol 2} with $(\rho,u,f,\rho_s)$ satisfying \eqref{condition on weak sol3}-\eqref{condition on weak sol5} and $(\bar\rho,\bar{u},\bar{f},\bar{\rho_s})$ satisfying \eqref{condition on weak sol6}-\eqref{condition on weak sol9}, if \eqref{condition on weak sol10} holds, and if all the norms occurring in the above conditions are bounded by $M$, then 
\begin{align}\label{bound on difference}
&\left(\int_0^T|u-\bar{u}|^2dxdt\right)^\frac{1}{2}+\sup_{0\le \tau\le T}\|(\rho-\bar{\rho})(\cdot,t)\|_{H^{-1}}\notag\\
&\qquad\le C\left[\|\rho_0-\bar{\rho_0}\|_{L^2\cap L^{2p}}+\|\rho_0u_0-\bar{\rho_0}\bar{u_0}\|_{L^2}\right]\notag\\
&\qquad\qquad+C\left[\left(\intox|\rho_s-\bar{\rho_s}|^2\right)^\frac{1}{2}+\left(\intox|\bar{f}-\bar{f}\circ S|^2dxdt\right)^\frac{1}{2}\right],
\end{align}
where $S$ is defined in \eqref{def of S} later. If we further have $\int_0^Tt\|\nabla\bar{f}(\cdot,t)\|_{L^\infty}\le M$, then $\bar{f}\circ S$ may be replaced by $\bar{f}$ in \eqref{bound on difference}.
\end{thm}
\begin{rem}
Similar to the case as in Hoff \cite{hoff06}, under a more general condition on $P$, namely
\begin{equation}\label{more general condition on P}
\sup_{0\le t\le T}\Big\|\nabla\Big(\frac{P(\rho(\cdot,t))-P(\bar{\rho}(\cdot,t))}{\rho(\cdot,t)-\bar{\rho}(\cdot,t)}\Big)\Big\|_{L^3}<\infty,
\end{equation}
one can still obtain the same conclusion \eqref{bound on difference} from Theorem~\ref{Main thm}.
\end{rem}

The rest of the paper is organised as follows. In section~\ref{prelim section}, we recall some known facts and useful estimates, and we further discuss the global-in-time existence and regularity of weak solution to \eqref{NS}-\eqref{IC}. In section~\ref{proof of main thm section}, we address the uniqueness of weak solutions given in Theorem~\ref{Main thm} by making use of the Lagrangean framework and bounds on the weak solutions.

\section{Existence and regularity of weak solution}\label{prelim section}

In this section, we state some known facts and estimates which will be useful for later analysis. We also discuss the global-in-time existence and regularity of weak solution to \eqref{NS}-\eqref{IC} which will be summarised in Theorem~\ref{existence of weak sol thm}. 

To begin with, we state the following Gagliardo-Nirenberg type inequalities and the proof can be found in Ziemer \cite{ziemer}:

\begin{prop}
For $p\in[2,6]$, $q\in(1,\infty)$ and $r\in(3,\infty)$, there exists some generic constant $C>0$ such that for any $f\in H^1$ and $g\in L^q\cap D^{1,r}$, we have
\begin{align}
\|f\|^p_{L^p}&\le C\|f\|^\frac{6-p}{2}_{L^2}\|\nabla f\|^\frac{3p-6}{2}_{L^2},\label{GN1}\\
\|g\|_{L^\infty}&\le C\|g\|^\frac{q(r-3)}{3r+q(r-3)}_{L^2}\|\nabla g\|^\frac{3r}{3r+q(r-3)}_{L^2}.\label{GN2}
\end{align}
\end{prop}

Next, we recall the following lemma which gives some useful estimates on $u$ in terms of $F$ and $\omega$.
\begin{lem}
For $r_1,r_2\in(1,\infty)$ and $t>0$, there exists a universal constant $C$ which depends only on $r_1$, $r_2$, $\mu$, $\lambda$, $a$, $\gamma$ and $\trho$ such that, the following estimates hold:
\begin{align}\label{bound on F and omega in terms of u}
\|\nabla F\|_{L^{r_1}}+\|\nabla\omega\|_{L^{r_1}}&\le C(\|\rho^\frac{1}{2} \dot{u}\|_{L^{r_1}}+\|(\rho-\trho)\|_{L^{r_1}})
\end{align}
\begin{align}\label{bound on u in terms of F and omega}
||\nabla u(\cdot,t)||_{L^{r_2}}\le C(|| F(\cdot,t)||_{L^{r_2}}+||\omega(\cdot,t)||_{L^{r_2}}+||(\rho-\rho_s )(\cdot,t)||_{L^{r_2}}).
\end{align}
\end{lem}
\begin{proof}
In view of the Poisson equations \eqref{elliptic eqn for F} and \eqref{elliptic eqn for omega}, we can apply standard $L^p$-estimate on $F$ and $\omega$ to obtain \eqref{bound on F and omega in terms of u}-\eqref{bound on u in terms of F and omega}; see \cite{suen13b, suen14, suen16}) for more details.
\end{proof}

Before we discuss the existence and some further properties of weak solution to \eqref{NS}-\eqref{IC}, we introduce the notion of {\it piecewise H\"{o}lder continuous} as follows (also refer to Hoff \cite{hoff02} for more details):

\begin{defn}\label{definition of piecewise C beta}
We say that a function $\phi(\cdot,t)$ is piecewise $C^{\beta(t)}$ if it has simple discontinuities across a $C^{1+\beta(t)}$ curve $\mathcal{C}(t):\mathcal{C}(t)=\{y(s,t):s\in I\subset\R\}$, where $\beta(t)>0$ is a function in $t$, $I$ is an open interval and the curve $\mathcal{C}(t)$ is the $u$-transport of $\mathcal C(0)$ given by:
$$y(s,t)=y(s,0)+\int_0^t u(y(s,\tau),\tau)d\tau.$$
Here $\mathcal{C}(0)$ is a $C^{\beta_0}$ curve with $\beta(0)=\beta_0>0$, which means that
\begin{equation*}
\mathcal{C}(0)=\{y_0(s):s\in\R\},
\end{equation*}
where $y(s,0)=y_0(s)$ is parameterised in arc length $s$ and $y_0$ is $C^{\beta_0}$.
\end{defn}

We now give the following theorem which gives the global-in-time existence and regularity of weak solution $(\rho,u,\rho_s)$ to \eqref{NS}-\eqref{IC}:

\begin{thm}\label{existence of weak sol thm}
Given $a>0$, let $P$, $f$, $\lambda$, $\mu$ be the system parameters in \eqref{NS} satisfying \eqref{condition on P}-\eqref{assumption on viscosity}. The system \eqref{NS}-\eqref{IC} has a global-in-time weak solution $(\rho,u,\rho_s)$ provided that
\begin{align}\label{smallness assumption} 
\left\{ \begin{array}{l}
\rho_0\ge0 \text{ a.e.},\qquad\rho_0\in L^\infty,\\
\dis{\intox\rho_0|u_0|^q dx<\infty},\\
\|\rho_0-\tilde\rho\|_{L^2}+\|\rho_0^\frac{1}{2}u_0\|_{L^{2}}\ll1,
\end{array}\right.
\end{align}
where $q>6$. The solution can be shown to satisfy conditions \eqref{condition on rho s}-\eqref{spaces for weak sol 2}, and if $\inf\rho_0>0$, then the solution further satisfies \eqref{condition on weak sol4}-\eqref{condition on weak sol5} and the energy estimates: for $u_0\in H^s$ for some $s\in[0,1]$, then it holds
\begin{align}\label{energy bound on weak sol}
&\sup_{0\le \tau\le T}\intox[\rho|u|^2+|\rho-\rho_s|^2+\tau^{1-s}|\nabla u|^2+\tau^\sigma|\dot{u}|^2)dx\notag\\
&\qquad+\int_0^T\intox(|\nabla u|^2+\tau^{1-s}|\dot{u}|^2+\tau^\sigma|\nabla\dot{u}|^2)dxd\tau\le C(T),
\end{align}
where $\sigma=\max\{2-s,3-3s\}$ and $C(T)$ is a generic positive constant which depends on $T$. Moreover, if $\rho_0$ is {\it piecewise} $C^{\beta_0}$ for some $\beta_0>0$ in the sense of Definition~\ref{definition of piecewise C beta}, then for each positive time $T$ and $t\in[0,T]$, there exists function $\beta(t)\in(0,\beta_0]$ such that $\rho$ is {\it piecewise} $C^{\beta(t)}$ on $[0,T]$.
\end{thm}

\begin{proof}
The existence result follows mainly from Suen \cite{suen13b, suen14, suen16}, which show that the global-in-time weak solution $(\rho,u,\rho_s)$ satisfies \eqref{condition on rho s}-\eqref{spaces for weak sol 2}. The bounds \eqref{condition on weak sol4}-\eqref{condition on weak sol5} and the energy estimates \eqref{energy bound on weak sol} for $s=0$ follow by \cite[Theorem~1.1]{suen16}, and the general cases for $s\in(0,1]$ can be proved by the interpolation techniques of Hoff \cite{hoff02} or Suen \cite{suen20b} provided that $\inf \rho_0>0$. 

To prove that $\rho$ is piecewise $C^{\beta(t)}$ on $[0,T]$ for the case when $\rho_0$ is piecewise $C^{\beta_0}$, it involves an argument which is based on the observation of ``enhanced regularity" gained by the effective viscous flux $F$. Details of the proof can be found in \cite{hoff02, suen20b, suen20mhd} and we only give a sketch here. We introduce a decomposition of $u$ which is given by $u=u_{F}+u_{P}$, where $u_{F}$, $u_{P}$ satisfy
\begin{align}\label{decomposition of u}
\left\{
 \begin{array}{lr}
(\mu+\lambda)\Delta (u_{F})^{j}=(\rho_sF)_{x_j} +(\mu+\lambda)(\omega)^{j,k}_{x_k}\\
(\mu+\lambda)\Delta (u_P)^{j}=(P-P_s)_{x_j}.\\
\end{array}
\right.
\end{align}
Using the estimates \eqref{bound on F and omega in terms of u} on $F$ and $\omega$, and together with \eqref{energy bound on weak sol}, we readily have
\begin{align}\label{bound on nabla uF}
\int_{0}^{T}||\nabla u_{F}(\cdot,\tau)||_{\infty}d\tau\le C(T).
\end{align}
On the other hand, in order to control $u_P$, by applying the results from Bahouri-Chemin \cite{BC94} on Newtonian potential, we can make use of the pointwise bounds \eqref{condition on weak sol4} on $\rho$ to show that $u_P$ is, in fact, log-Lipschitz with bounded log-Lipschitz seminorm. This is sufficient to guarantee that the integral curve $x(y,t)$ as defined by 
\begin{align*}
\left\{ \begin{array}
{lr} \dot{x}(t)
=u(x(t),t)\\ x(0)=y,
\end{array} \right.
\end{align*}
is H\"{o}lder-continuous in $y$. Upon integrating the mass equation along integral curves $x(t,y)$ and $x(t,z)$, subtracting and recalling the definition \eqref{def of F and omega} of $F$, we arrive at 
\begin{align}\label{Holder norm estimate of rho}
\log\rho&(x(T,y),T)-\log\rho(x(T,z),T)\notag\\
&=\log\rho_0(y)-\log\rho_0(z)
+\int_0^T [P(\rho(x(\tau,y),\tau)-P(\rho(x(\tau,z),\tau)]d\tau\notag \\
&\qquad\qquad\qquad\qquad\qquad\qquad\quad+\int_0^t[F(x(\tau,y),\tau)-F(x(\tau,z),\tau)]d\tau.
\end{align}
Since $P$ is increasing, the second term on the right side of the above can be dropped out. Moreover, with the help of the estimate \eqref{bound on F and omega in terms of u} on $F$ and the H\"{o}lder-continuity of $x(y,t)$, the third term can be bounded by $M$. Hence we can conclude from \eqref{Holder norm estimate of rho} that $\rho(\cdot,t)$ is $C^{\beta(t)}$ on $[0,T]$ for some $\beta(t)\in(0,\beta_0]$ with bounded modulus.
\end{proof}
\begin{rem}\label{explanation on the time integral bound on nabla u}
Under the assumption that the initial density $\rho_0$ is piecewise H\"{o}lder continuous, by Theorem~\ref{existence of weak sol thm}, we can see that $\rho$ is piecewise H\"{o}lder continuous. As a consequence, it further implies that $\nabla u\in L^1((0,T);W^{1,\infty})$. To see how it works, we make use of the Poisson equation \eqref{decomposition of u}$_2$ again and apply properties of Newtonian potentials to conclude that the $C^{1+\beta(t)}(\R^3)$ norm of $u_P$ remains finite in finite time, hence the following bound holds for $\nabla u_P$ as well:
\begin{align}\label{bound on nabla uP}
\int_{0}^{T}||\nabla u_{P}(\cdot,\tau)||_{\infty}d\tau\le C(T).
\end{align}
Together with the bound \eqref{bound on nabla uF} on $u_F$, we conclude that condition \eqref{spaces for weak sol 1} holds for the weak solution to \eqref{NS}-\eqref{IC} with piecewise H\"{o}lder continuous initial density. The results of Theorem~\ref{Main thm} therefore do apply to this class of weak solutions, which includes solutions with Riemann-like initial data.
\end{rem}

We end this section by stating some results on Lagrangean structure which will be used in this paper. As suggested by Hoff in \cite{hoff06}, weak solutions with minimal regularity are best compared in a Lagrangian framework. In other words, we try to compare the instantaneous states of corresponding fluid particles in two different solutions. To achieve our goal, we employ some delicate estimates on {\it particle trajectories}. More precisely, for $T>0$, the bound \eqref{bound on nabla uF} and \eqref{bound on nabla uP} guarantee the existence and uniqueness of the mapping $X(y,t,t')\in C(\R^3\times[0,T]^2)$ satisfying
\begin{align}\label{integral curve X}
\left\{ \begin{array}
{lr} \dis\frac{\partial X}{\partial t}(y,t,t')
=u(X(y,t,t'),t)\\ X(y,t',t')=y
\end{array} \right.
\end{align}
where $(\rho,u,B)$ is a weak solution to \eqref{NS}-\eqref{IC}. Moreover, the mapping $X(\cdot,t,t')$ is Lipschitz on $\R^3$ for $(t,t')\in[0,T]^2$. The results are given in the following proposition and the proof can be found in Hoff \cite{hoff06}.
\begin{proposition}\label{prop on X}
Let $T>0$ and $u$ satisfy \eqref{spaces for weak sol 1}. Then there is a unique function $X\in C(\R^3\times[0,T]^2$) satisfying \eqref{integral curve X}. In particular, $X(\cdot,t,t')$ is Lipschitz on $\R^3$ for $(t,t')\in[0,T]^2$, and there is a constant $C$ such that
\begin{equation*}
\Big\|\frac{\partial X}{\partial y}(\cdot,t,t')\Big\|_{L^\infty}\le C,\qquad (t,t')\in[0,T]^2.
\end{equation*}
\end{proposition} 
With respect to velocities $u$ and $\bar{u}$, for $y\in\R^3$, we let $X$, $\bar{X}$ be two integral curves given by
\begin{align*}
\left\{ \begin{array}
{lr} \dis\frac{\partial X}{\partial t}(y,t,t')
=u(X(y,t,t'),t)\\ X(y,t',t')=y
\end{array} \right.
\end{align*}
and
\begin{align*}
\left\{ \begin{array}
{lr} \dis\frac{\partial \bar{X}}{\partial t}(y,t,t')
=\bar{u}(\bar{X}(y,t,t'),t)\\ \bar{X}(y,t',t')=y.
\end{array} \right.
\end{align*}
We then define $S(x,t)$, $S^{-1}(x,t)$ by
\begin{equation}\label{def of S}
S(x,t)=\bar{X}(X(x,0,t),t,0),
\end{equation}
and 
\begin{equation}\label{def of S-1}
S^{-1}(x,t)=X(\bar{X}(x,0,t),t,0).
\end{equation}
The following proposition provides some properties of $S$ and $S^{-1}$, which will be crucial for later analysis.
\begin{proposition}\label{prop on S}
Let $S$ and $S^{-1}$ be as given in \eqref{def of S}-\eqref{def of S-1}. Then we have
\begin{itemize}
\item $S^{\pm1}$ is continuous on $R^3\times[0,T]$ and Lipschitz continuous on $R^3\times[\tau,T]$ for all $\tau>0$, and there is a constant C such that $$\|\nabla S^{\pm1}(\cdot,t)\|_{L^\infty}\le C,\qquad t\in[0,T];$$
\item $(S_t+\nabla S u)(x,t)=\bar{u}(S(x,t),t)$ a.e. in $\R^3\times(0,T);$
\item $\bar{\rho}(S(x,t),t)\rho_0(X(x,0,t))\det\nabla S(x,t)=\rho(x,t)\bar{\rho}_0(X(x,0,t))$ a.e. in $\R^3\times(0,T)$;
\item If $u, u \in L^2(\R^3 \times (0,T))$, then for all $t\in(0,T)$,
\begin{align}
\intox|x-S(x,t)|^2dx&\le Ct\int_0^t\intox|u(x,\tau)-\bar{u}(S(x,\tau),\tau)|^2dxd\tau,\label{bound on S}\\
\intox|x-S^{-1}(x,t)|^2dx&\le Ct\int_0^t\intox|u(S^{-1}(x,\tau),\tau)-\bar{u}(x,\tau)|^2dxd\tau.\label{bound on S -1}
\end{align}
\end{itemize}
\end{proposition}
\begin{proof}
The proof can be found on page 1752 in \cite{hoff06}.
\end{proof}

\section{Proof of Theorem~\ref{Main thm}}\label{proof of main thm section}

We are now ready to give the proof of Theorem~\ref{Main thm}. Throughout this section, $C$ always denotes a generic positive constant which depends on the parameters $P$, $f$, $\lambda$, $\mu$, $T$, $a$, $r$ as described in Theorem~\ref{Main thm}. First, we let $\psi:\R^3\times[0,T]\rightarrow\R^3$ be a test function satisfying
\begin{align}\label{weak form for rho}
&-\intox \rho_0(x)u_0(x)\psi(x,0)dx\notag\\
&=\int_0^T\intox \Big[\rho u\cdot(\psi_t+\nabla\psi u)+(P(\rho)-P_s)\divv(\psi)-\mu\nabla u^j\cdot\nabla\psi^j\notag\\
&\qquad\qquad\qquad-\lambda\divv(u)\divv(\psi)-\lambda(\divv(u)\divv(\psi)+(\rho - \rho_s)f\cdot\psi\Big]dxd\tau,
\end{align}
where we used the expressions \eqref{condition on f} and \eqref{eqn for steady state}$_1$ for $\nabla P_s$. Define $\bar\psi=\psi\circ S^{-1}$. Then we have
\begin{align}\label{weak form for bar rho}
&-\intox \bar{\rho}_0(x)\bar{u}_0(x)\bar{\psi}(x,0)dx\notag\\
&=\int_0^T\intox \Big[\bar{\rho}\bar{u}\cdot(\bar{\psi}_t+\nabla\bar{\psi}\bar{u})+(P(\bar{\rho})-\bar{P_s})\divv(\bar{\psi})-\mu\nabla\bar{u}^j\cdot\nabla\bar{\psi}^j\notag\\
&\qquad\qquad\qquad-\lambda\divv(\bar{u})\divv(\bar{\psi})-\lambda(\divv(\bar{u})\divv(\bar{\psi})+(\bar{\rho} - \bar{\rho_s})\bar{f}\cdot\bar{\psi}\Big]dxd\tau,
\end{align}
with $\bar{P_s}=P(\bar{\rho_s})$. Notice that using the definition of $\bar{F}$ and $\omega$ from \eqref{def of F and omega} (replacing $F$ by $\bar{F}$, $u$ by $\bar{u}$, etc.), 
\begin{align*}
&\intoxt\left[(\bar{P_s}-\bar{P})\divv(\bar\psi)+\mu\nabla\bar{u}^j\cdot\nabla\bar{\psi}^j+\lambda\divv(\bar{u})\divv(\bar{\psi})\right]\\
&=\intoxt\left[(\mu+\lambda)\divv(\bar{u})-\bar{P}+\bar{P_s}\right]\divv(\bar{\psi})+\intoxt\mu(\bar{u}^j_{x_k}-\bar{u}^k_{x_j})\bar{\psi}^j_{x_k}\\
&=-\intoxt[\nabla(\bar{\rho_s}\bar{F})\cdot\psi+\mu\bar{\omega}^{j,k}_{x_k}\psi^j]\\
&\qquad+\int_0^T\intox \left[\nabla(\bar{\rho_s}\bar{F})\cdot(\psi-\psi\circ S^{-1})+\mu\bar{\omega}^{j,k}_{x_k}(\psi^j-\psi^j\circ S^{-1})\right]\\
&=\int_0^T\intox \left[(P_s-P)\divv(\psi)+\mu\nabla\bar{u}^j\cdot\nabla\psi^j+\lambda\divv(\bar{u})\divv(\psi)\right]\\
&\qquad+\int_0^T\intox \left[\nabla(\bar{\rho_s}\bar{F})\cdot(\psi-\psi\circ S^{-1})+\mu\bar{\omega}^{j,k}_{x_k}(\psi^j-\psi^j\circ S^{-1})\right].
\end{align*}
Moreover, we have
\begin{align*}
\intox \bar{\rho}\bar{u}\cdot(\bar{\psi}_t+\nabla\bar{\psi}\bar{u})dx&=\intox \bar{\rho}(S)\bar{u}(S)\cdot(\bar{\psi}_t(S)+\nabla\bar{\psi}\bar{u}(S))|\det(\nabla S)|dx\\
&=\intox A_0\rho\bar{u}(S)(\psi_t+\nabla\psi u)dx,
\end{align*}
where we used the fact that $A_0\rho=(\bar{\rho}\circ S)|\det(\nabla S)|$ from Proposition~\ref{prop on S}. Hence by taking the difference between \eqref{weak form for rho} and \eqref{weak form for bar rho}, for all $\psi$ and $\bar{\psi}$, we have
\begin{align}\label{estimate on weak form for u and rho}
&\intox (\bar{\rho}\bar{u}_0-\rho_0u_0)\cdot\psi(x,0)dx\notag\\
&=\int_0^T\intox \left[\rho(u-\bar{u}\circ S)(\psi_t+\nabla\psi u)+(1-A_0)\rho(\bar{u}\circ S)(\psi_t+\nabla\psi u)\right]\\
&\qquad+\int_0^T\intox \left[(P_s-P)\divv(\psi)+\mu\nabla\bar{u}^j\cdot\nabla\psi^j+\lambda\divv(\bar{u})\divv(\psi)\right]\notag\\
&\qquad+\int_0^T\intox \left[(1-A_0)\rho(\bar{u}\circ S)(\psi_t+\nabla\psi u)\right]\notag\\
&\qquad+\int_0^T\intox \left[\nabla(\bar{\rho_s}\bar{F})\cdot(\psi-\psi\circ S^{-1})+\mu\bar{\omega}^{j,k}_{x_k}(\psi^j-\psi^j\circ S^{-1})\right]\notag\\
&\qquad+\int_0^T\intox (\bar{u}\circ S-\bar{u})(\mu\Delta\psi+\lambda\nabla\divv(\psi))\notag\\
&\qquad+\int_0^T\intox (P-\bar{P})\divv(\psi)+\int_0^T\intox (\bar{P_s}-P_s)\divv(\psi)\notag\\
&\qquad+\int_0^T\intox \left(\bar{\rho_s}(\bar{f}\circ S)\cdot\psi-\rho_s f\cdot\psi\right)\notag.
\end{align}
Next we extend $\rho$, $u$ to be constant in $t$ outside $[0,T]$ and let $\rho^\varepsilon$ and $u^\varepsilon$ be the corresponding smooth approximation obtained by mollifying in both $x$ and $t$. Then we define $\psi^\varepsilon:\R^3\times[0,T]\to\R^3$ to be the solutions satisfying
\begin{align*}
\left\{ \begin{array}
{lr} \rho^\varepsilon(\psi^\varepsilon_t+u^\varepsilon\cdot\nabla\psi^\varepsilon)+\mu\Delta\psi^\varepsilon+\lambda\nabla\divv(\psi^\varepsilon)=G\\ 
\psi^\varepsilon(\cdot,T)=0.
\end{array} \right.
\end{align*}
By simple estimates (or refer to \cite[Lemma~3.1]{hoff06}), $\psi^\varepsilon$ satisfies the following bounds in terms of $G$
\begin{align}\label{bound on psi 1}
&\sup_{0\le \tau\le T}\intox[|\psi^\varepsilon(x,\tau)|^2+|\nabla\psi^\varepsilon(x,\tau)|^2]+\int_0^T\intox[|\psi^\varepsilon_t+\nabla\psi^\varepsilon u^\varepsilon|^2+|D^2_x\psi^\varepsilon|^2]\notag\\
&\le C\int_0^T\intox|G|^2,
\end{align}
\begin{align}\label{bound on psi 2}
\sup_{0\le \tau\le T}\|\psi^\varepsilon(\cdot,\tau)\|_{L^\infty}+\int_0^T\intox|\psi^\varepsilon|^r\le C(G),
\end{align}
for some positive constant $C(G)$ which depends on $G$. We now take $\psi=\psi^\varepsilon$ in \eqref{estimate on weak form for u and rho} to obtain
\begin{align}\label{estimate on weak form for u and rho 2}
\intox (\bar{\rho}\bar{u}_0-\rho_0u_0)\cdot\psi^\varepsilon(x,0)dx=\intoxt z\cdot G+\sum_{i=1}^7\mathcal{R}_i,
\end{align}
where $z=u-\bar{u}\circ S$ and $\mathcal{R}_1,\dots,\mathcal{R}_7$ are given by:
\begin{align*}\
&\mathcal{R}_1=\intoxt\left[\nabla(\bar{\rho_s}\bar{F})\cdot(\psi^\varepsilon-\psi^\varepsilon\circ S^{-1})+\mu\bar{\omega}^{j,k}_{x_k}(\psi^\varepsilon-\psi^\varepsilon\circ S^{-1})\right],\\
&\mathcal{R}_2=\intoxt \left[\rho(f-\bar{f}\circ D)\cdot\psi^\varepsilon+(1-A_0)\rho(\bar{f}\circ S)\cdot\psi^\varepsilon\right],\\
&\mathcal{R}_3=\intoxt(\bar{u}\circ S-\bar{u})\cdot(\mu\Delta\psi^\varepsilon+\lambda\divv(\psi^\varepsilon)),\\
&\mathcal{R}_4=\intoxt z\cdot \left[(\rho-\rho^\varepsilon)\psi^\varepsilon_t+\nabla\psi^\varepsilon(\rho u-\rho^\varepsilon u^\varepsilon)\right],\\
&\mathcal{R}_5=\intoxt(1-A_0)\rho(\bar{u}\circ S)\cdot(\psi^\varepsilon_t+\nabla\psi^\varepsilon u),\\
&\mathcal{R}_6=\intoxt(P-\bar{P})\divv(\psi^\varepsilon)+\int_0^T\intox (\bar{P_s}-P_s)\divv(\psi^\varepsilon)\notag\\
&\mathcal{R}_7=\int_0^T\intox \left(\bar{\rho_s}(\bar{f}\circ S)\cdot\psi^\varepsilon-\rho_s f\cdot\psi^\varepsilon\right)\notag.
\end{align*}
The left side of \eqref{estimate on weak form for u and rho 2} can be readily bounded by 
\begin{align}\label{estimate of LHS}
\Big|\intox (\bar{\rho}\bar{u}_0-\rho_0u_0)\cdot\psi^\varepsilon(x,0)dx\Big|\le\|\rho_0u_0-\bar{\rho}_0\bar{u}_0\|_{L^2}\left(\intoxt|G|^2\right)^\frac{1}{2}.
\end{align}
Following the same method given in \cite{hoff06} and with the help of the bounds \eqref{bound on psi 1}-\eqref{bound on psi 2}, the terms $\mathcal{R}_2$, $\mathcal{R}_3$, $\mathcal{R}_4$ and $\mathcal{R}_5$ can be estimated as follows:
\begin{align}\label{estimate on R2}
|\mathcal{R}_2|\le C\left[\left(\intoxt|f-\bar{f}\circ S|^2\right)^\frac{1}{2}+\|\rho_0-\bar{\rho_0}\|_{L^2{2p}}\right]\left(\intoxt|G|^2\right)^\frac{1}{2},
\end{align}
\begin{align}\label{estimate on R3}
|\mathcal{R}_3|\le C\left(\intoxt|z|^2\right)^\frac{1}{2}\left(\intoxt|G|^2\right)^\frac{1}{2},
\end{align}
\begin{align}\label{estimate on R4}
\lim_{\varepsilon\to0}\mathcal{R}_4=0,
\end{align}
and
\begin{align}\label{estimate on R5}
|\mathcal{R}_5|\le C\left[\|\rho_0-\bar{\rho}_0\|_{L^2}+\left(\intoxt|z|^2\right)^\frac{1}{2}\right]\left(\intoxt|G|^2\right)^\frac{1}{2}.
\end{align}
It remains to estimate the terms $\mathcal{R}_1$, $\mathcal{R}_6$ and $\mathcal{R}_7$. For $\mathcal{R}_1$, using H\"{o}lder inequality, we have
\begin{align*}
|\mathcal{R}_1|&\le C\left(\intoxt|z|^2\right)^\frac{1}{2}\int_0^T t^\frac{1}{2}\|\nabla(\bar{\rho_s}\bar{F})(\cdot,t)\|_{L^4}\|\nabla\psi^\varepsilon(\cdot,t)\|_{L^4}dt\\
&\qquad+C\left(\intoxt|z|^2\right)^\frac{1}{2}\int_0^T t^\frac{1}{2}\|\nabla\bar{\omega}(\cdot,t)\|_{L^4}\|\nabla\psi^\varepsilon(\cdot,t)\|_{L^4}dt\\
&\le C\left(\intoxt|z|^2\right)^\frac{1}{2}\left(\intoxt|G|^2\right)^\frac{1}{2}\left(\intoxt|D^2_x\psi^\varepsilon|^2\right)^\frac{3}{8}\\
&\qquad\times\left(\int_0^T t^\frac{4}{5}(\|\nabla(\bar{\rho_s}\bar{F})(\cdot,t)\|^\frac{8}{5}_{L^4}+\|\nabla\bar{\omega}(\cdot,t)\|^\frac{8}{5}_{L^4})dt\right)^\frac{5}{8}.
\end{align*}
Using \eqref{bound on F and omega in terms of u} and the boundedness assumption \eqref{condition on weak sol6} on $\bar{\rho_s}$, the term involving $\bar{F}$ and $\bar{\omega}$ can be bounded by
\begin{align*}
\intox t^\frac{4}{5}\left(\intox|\dot{\bar{u}}|^4\right)^\frac{3}{5}+C
\end{align*}
and with the help of \eqref{GN1} and the energy estimates \eqref{energy bound on weak sol}, we further have
\begin{align*}
&\int_0^T t^\frac{4}{5}\left(\intox|\dot{\bar{u}}|^4\right)^\frac{3}{5}\\
&\le C\int_0^T t^\frac{4}{5}\left(\intox|\dot{\bar{u}}|^2\right)^\frac{1}{5}\left(\intox|\nabla\dot{\bar{u}}|^2\right)^\frac{3}{5}\\
&\le C\left(\int_0^T t^{4s-3}\right)^\frac{1}{5}\left(\int_0^T t^{1-s}\intox|\dot{\bar{u}}|^2\right)^\frac{1}{5}\left(\int_0^T t^{1-s}\intox|\dot{\bar{u}}|^2\right)^\frac{1}{5}\le CT^\frac{4s-2}{5}.
\end{align*}
Hence we conclude
\begin{align}\label{rough bound on R1}
|\mathcal{R}_1|\le CT^\frac{2s-1}{4}\left(\intoxt|z|^2\right)^\frac{1}{2}\left(\intoxt|G|^2\right)^\frac{1}{2}.
\end{align}
In particular, for $[t_1,t_2]\subseteq[0,T]$, if we define
\begin{equation*}
\mathcal{R}_1(t_1,t_2)=\int_{t_1}^{t_2}\intox\left[\nabla(\bar{\rho_s}\bar{F})\cdot(\psi^\varepsilon-\psi^\varepsilon\circ S^{-1})+\mu\bar{\omega}^{j,k}_{x_k}(\psi^\varepsilon-\psi^\varepsilon\circ S^{-1})\right],
\end{equation*}
then we also have
\begin{equation}\label{estimate on R1}
|\mathcal{R}_1(t_1,t_2)|\le C|t_2-t_1|^\frac{2s-1}{4}\left(\intoxts|z|^2\Big)^\frac{1}{2}\Big(\intoxts|G|^2\right)^\frac{1}{2}.
\end{equation}

For $\mathcal{R}_6$, using the assumption \eqref{condition on P} on $P$, we have \footnote{We point out that the bound \eqref{H -1 bound on P} also holds under the more general condition \eqref{more general condition on P} on the pressure; see \cite{hoff06} for a more detailed proof.}
\begin{align}\label{H -1 bound on P}
\|(P-\bar{P})(\cdot,t)\|_{H^{-1}}\le C\|(\rho-\bar{\rho})(\cdot,t)\|_{H^{-1}}.
\end{align}
Hence together with the bound \eqref{bound on psi 2} on $\psi^\varepsilon$, it implies that
\begin{align}\label{bound on R6 1}
\Big|\intoxt(P-\bar{P})\divv(\psi^\varepsilon)\Big|&\le\int_0^T\|(P-\bar{P})(\cdot,t)\|_{H^{-1}}\|\divv(\psi^\varepsilon)\|_{H^1}d\tau\notag\\
&\le C\sup_{0\le t\le T}\|(\rho-\bar{\rho})(\cdot,t)\|_{H^{-1}}\left(\intoxt|G|^2\right)^\frac{1}{2}.
\end{align}
Following the argument given in \cite{hoff06}, the term $\sup_{0\le t\le T}\|(\rho-\bar{\rho})(\cdot,t)\|_{H^{-1}}$ can be bounded by
\begin{align*}
\sup_{0\le t\le T}\|(\rho-\bar{\rho})(\cdot,t)\|_{H^{-1}}\le C\left[\|\rho_0-\bar{\rho_0}\|_{L^2}+T^\frac{1}{2}\left(\intoxt|z|^2\right)^\frac{1}{2}\right],
\end{align*}
and we conclude from \eqref{bound on R6 1} that
\begin{align}\label{bound on R6 2}
\Big|\intoxt(P-\bar{P})\divv(\psi^\varepsilon)\Big|\le C\left[\|\rho_0-\bar{\rho_0}\|_{L^2}+\left(\intoxt|z|^2\right)^\frac{1}{2}\right]\left(\intoxt|G|^2\right)^\frac{1}{2}.
\end{align}
The term $\int_0^T\intox (\bar{P_s}-P_s)\divv(\psi^\varepsilon)$ can be readily bounded by
\begin{align}\label{bound on R6 3}
\Big|\int_0^T\intox (\bar{P_s}-P_s)\divv(\psi^\varepsilon)\Big|\le C\left(\intox|\rho_s-\bar{\rho_s}|^2\right)^\frac{1}{2},
\end{align}
and therefore the bounds \eqref{bound on R6 2}-\eqref{bound on R6 3} give
\begin{align}\label{estimate on R6}
|\mathcal{R}_6|&\le C\left[\|\rho_0-\bar{\rho_0}\|_{L^2}+\left(\intoxt|z|^2\right)^\frac{1}{2}\right]\left(\intoxt|G|^2\right)^\frac{1}{2}\notag\\
&\qquad+\left(\intox|\rho_s-\bar{\rho_s}|^2\right)^\frac{1}{2}.
\end{align}
Finally, for the term $\mathcal{R}_7$, we can rewrite it as follows.
\begin{align*}
\mathcal{R}_7=\intoxt(\bar{\rho_s}-\rho_s)f\cdot\psi^\varepsilon+\intoxt\bar{\rho_s}(\bar{f}\circ S-f)\cdot\psi^\varepsilon
\end{align*}
The term $\intoxt(\bar{\rho_s}-\rho_s)f\cdot\psi^\varepsilon$ can be bounded by
\begin{align*}
\Big|\intoxt(\bar{\rho_s}-\rho_s)f\cdot\psi^\varepsilon\Big|&\le C\|f\|_{L^\infty}\left(\intox|\rho_s-\bar{\rho_s}|^2\right)^\frac{1}{2}\left(\intoxt|\psi^\varepsilon|^2\right)^\frac{1}{2}\\
&\le C\|f\|_{L^\infty}\left(\intox|\rho_s-\bar{\rho_s}|^2\right)^\frac{1}{2}\left(\intoxt|G|^2\right)^\frac{1}{2},
\end{align*}
and similarly, $\intoxt\bar{\rho_s}(\bar{f}\circ S-f)\cdot\psi^\varepsilon$ can be bounded by
\begin{align*}
\intoxt\bar{\rho_s}(\bar{f}\circ S-f)\cdot\psi^\varepsilon\le C\|\bar{\rho_s}\|_{L^\infty}\left(\intox|\bar{f}\circ S-f|^2\right)^\frac{1}{2}\left(\intoxt|G|^2\right)^\frac{1}{2}.
\end{align*}
Recalling the assumptions \eqref{condition on weak sol3} and \eqref{condition on weak sol6}, we therefore obtain
\begin{align}\label{estimate on R7}
|\mathcal{R}_7|&\le C\left(\intox|\rho_s-\bar{\rho_s}|^2\right)^\frac{1}{2}\left(\intoxt|G|^2\right)^\frac{1}{2}\notag\\
&\qquad+C\left(\intox|\bar{f}\circ S-f|^2\right)^\frac{1}{2}\left(\intoxt|G|^2\right)^\frac{1}{2}.
\end{align}
Combining the estimates \eqref{estimate of LHS}, \eqref{estimate on R2}, \eqref{estimate on R3}, \eqref{estimate on R4}, \eqref{estimate on R5}, \eqref{estimate on R1}, \eqref{estimate on R6} and \eqref{estimate on R7}, we arrive at
\begin{align}\label{estimate on z in T}
\Big|\intoxt z\cdot G\Big|\le C\Big[M_0\Big(\intoxt|G|^2\Big)^\frac{1}{2}+|\mathcal{R}_1(0,T)|\Big],
\end{align}
where $M_0$ is given by
\begin{align*}
M_0&=\|\rho_0-\bar{\rho}_0\|_{L^2}+\|\rho_0u_0-\bar{\rho}_0\bar{u}_0\|_{L^2}+T^{\delta}\Big(\intoxt|z|^2\Big)^\frac{1}{2}\\
&\qquad+\left(\intox|\rho_s-\bar{\rho_s}|^2\right)^\frac{1}{2}+\left(\intox|\bar{f}-\bar{f}\circ S|^2dxdt\right)^\frac{1}{2}
\end{align*}
for some $\delta>0$, and $C>0$ is now fixed. Following the analysis given on page 1758-1759 in Hoff \cite{hoff06}, there exists a small time $\tilde\tau>0$ such that 
\begin{equation*}
\Big(\int_0^{\tilde\tau}\!\!\!\intox|z|^2\Big)^\frac{1}{2}\le 2CM_0,
\end{equation*}
and consequently
\begin{equation*}
|\mathcal{R}_1(0,\tilde\tau)|\le M_0\Big(\int_0^{\tilde\tau}\!\!\!\intox|G|^2\Big)^\frac{1}{2}.
\end{equation*}
By applying \eqref{estimate on z in T} with $T$ replaced by $2\tilde\tau$, we get
\begin{equation*}
\Big(\int_0^{2\tilde\tau}\!\!\!\intox|z|^2\Big)^\frac{1}{2}\le 4CM_0.
\end{equation*}
Since $\tilde\tau>0$ is fixed, we can exhaust the interval $[0,T]$ in finitely many steps to obtain that
\begin{equation*}
\Big(\intoxt|z|^2\Big)^\frac{1}{2}\le CM_0,
\end{equation*}
for some new constant $C>0$. Hence the term $T^{\delta}\Big(\intoxt|z|^2\Big)^\frac{1}{2}$ can be eliminated from the definition of $M_0$ by a Gronw\"{a}ll-type argument. Therefore we conclude that
\begin{align}
\Big|\intoxt z\cdot G\Big|&\le CM_0\Big(\intoxt|G|^2\Big)^\frac{1}{2}.\label{L2 bound on u uniqueness}
\end{align}
Since the bound \eqref{L2 bound on u uniqueness} holds for any $G\in H^\infty(\R^3\times[0,T])$, it shows that $\|z\|_{L^2([0,T]\times\R^3)}$ can be bounded by $M_0$. Finally, using the bound \eqref{spaces for weak sol 1} on the time integral on $\|\nabla\bar{u}\|_{L^\infty}$,
\begin{align*}
\intoxt|\bar{u}-\bar{u}\circ S|^2&\le\int_0^T\|\nabla\bar{u}(\cdot,t)\|^2_{L^\infty}\intox|x-S(x,t)|^2\\
&\le C\intoxt|z|^2,
\end{align*}
and hence \eqref{bound on difference} follows. This finishes the proof of Theorem~\ref{Main thm}.



\bibliographystyle{amsalpha}

\bibliography{References_for_uniqueness_large_potential}

\end{document}